\documentclass[11pt,a4paper,draft]{article}

\usepackage{amsmath,amssymb,amsthm,amscd}
\usepackage[]{fontenc}
\usepackage{xy}
\usepackage{enumerate}

\xyoption{all} \numberwithin{equation}{section}
\newtheorem{theorem}[equation]{Theorem}
\newtheorem{corollary}[equation]{Corollary}
\newtheorem{proposition}[equation]{Proposition}

\theoremstyle{definition}

\newtheorem{example}[equation]{Example}

\newtheorem{remark}[equation]{Remark}

\newcommand{\nc}{\newcommand}
\nc{\cH}{\mathcal{H}} \nc{\cA}{\mathcal{A}} \nc{\cG}{\mathcal{G}}
\nc{\cC}{\mathcal{C}} \nc{\cO}{\mathcal{O}} \nc{\cI}{\mathcal{I}}
\nc{\cB}{\mathcal{B}} \nc{\cY}{\mathcal{Y}} \nc{\cK}{\mathcal{K}}
\nc{\cX}{\mathcal{X}} \nc{\cS}{\mathcal{S}} \nc{\cE}{\mathcal{E}}
\nc{\cF}{\mathcal{F}} \nc{\cZ}{\mathcal{Z}} \nc{\cQ}{\mathcal{Q}}
\nc{\cN}{\mathcal{N}} \nc{\cP}{\mathcal{P}} \nc{\cL}{\mathcal{L}}
\nc{\cM}{\mathcal{M}} \nc{\cR}{\mathcal{R}} \nc{\cT}{\mathcal{T}}
\nc{\cW}{\mathcal{W}} \nc{\cU}{\mathcal{U}} \nc{\cD}{\mathcal{D}}
\nc{\cJ}{\mathcal{J}} \nc{\cV}{\mathcal{V}}
\nc{\fr}{{\rightarrow}}

\newcommand{\Q}{\mathbb{Q}}
\newcommand{\Z}{\mathbb{Z}}

\newcommand{\cp}{\mathbb{C}}
\newcommand{\pr}{\mathbb P}
\newcommand{\F}{\mathbb F}
\newcommand{\sym}{\mbox{\upshape{Sym}}}
\newcommand{\of}{\omega_f}
\newcommand{\rk}{\mbox{\upshape{rank}}}

\newcommand{\om}{\omega}

\newcommand{\oo}{\mathcal O}

\newcommand{\pic}{\mbox{\upshape{Pic}}}

\title{Slopes of trigonal fibred surfaces and of higher dimensional fibrations }
\author{M.A. Barja \footnote{Partially supported by MEC-MTM2006-14234-C02-02 and by Generalitat de Catalunya 2005SGR00557} , L. Stoppino \footnote{Partially
supported by PRIN 2005 ``Spazi di Moduli e Teorie di Lie'', FAR 2008 (Pavia) \textit{Variet\`a algebriche, calcolo algebrico, grafi orientati e topologici}}}

\begin{document}

\maketitle

\begin{abstract}
We give lower bounds for the slope of higher dimensional
fibrations $f\colon X \longrightarrow B$ over curves under conditions
of GIT-semistability of the fibres, using a generalization of a
method of Cornalba and Harris.
With the same method we establish a sharp
lower bound for the slope of trigonal fibrations of even genus and
general Maroni invariant; in particular this result proves a conjecture due to
 Harris and Stankova-Frenkel.
\end{abstract}

\medskip

 {\small {\it MSC 2000:} Primary 14D06, Secondary 14J10, 14H10}

 \bigskip

\section{Introduction and preliminaries}

Given a fibration over a curve $f\colon X
\longrightarrow B$ ($X,B$ complex, projective varieties, $B$ a
smooth curve, $f$ surjective with connected fibres) and a line
bundle ${\cal L}={\cal O}_X(L)$ on $X$, we can define the {\it
slope of the pair $(f,\cal L)$} to be the quotient

$$s(f,{\cal L})=\frac{L^n}{{\rm deg}{f_*{\cal L}}}$$

\noindent provided ${\rm deg}{f_*{\cal L}} \neq 0$, where $n$ is
the dimension of $X$.
When ${\cal L}=\omega_f$, the relative dualizing sheaf of $f$, we
simply call it the {\it slope} of $f$ and will denote it as
$s(f)$. Lower bounds for the slope have been extensively studied
in the literature (e.g., \cite{A-K}, \cite{B-S}, \cite{B-Z},
\cite{LS}, \cite{X}, \cite{konnocliff}, \cite{konnotrig}) for the
case of fibred surfaces ($n=2$) and some results are known in
dimension $n=3$ (\cite{Barja3folds}, \cite{ohno}).

In this paper, we study this problem using a generalized version
of a theorem of Cornalba-Harris (\cite{C-H},
\cite{LS}).
This method provides  a general result to produce lower bounds of
$s(f,\cal L)$ provided the pair $(F, |{\cal L}_{|F}|)$ (where $F$
is a general fibre of $f$) is semistable in the sense of Geometric Invariant Theory.
In Chapter 2 we recall this result and derive the following consequences
(see corollaries \ref{CH2} and \ref{CH3}
for a more detailed statement).

\begin{theorem}
Under the above hypotheses, assume $RŒf_*{\cal L}^h=0$ for $i>0$
and $h\gg0$

\begin{itemize}

\item[i)] If $|{\cal L}_{|F}|$ induces an embedding, we have

$$
L^n \geq n\frac{(L_{\vert F})^{n-1}}{h^0(F,\cL_{\vert F})}\deg
f_*\cL. $$

\item[ii)] If $|{\cal L}_{|F}|$ induces an generically finite rational map onto a variety of degree $d$ and $f_*{\cal L}^h$
 is nef, we have

$$
L^n \geq n\frac{d}{h^0(F,\cL_{\vert F})}\deg f_*\cL. $$

\end{itemize}

\end{theorem}

\smallskip

The method applied directly to families of canonical varieties
would give very interesting higher dimensional slope inequalities.
However, already in the case of surfaces it is very hard to check the
stability assumption. For instance it is not known if a
``general''  surface of general type satisfies it or not.
In the case of  hypersurfaces with high enough degree,
and with log-terminal singularities the stability has been proven  by  Tian
(\cite{Tian}) using methods of differential geometry.
 Then we can deduce a bound on the slope, when the
fibres are hyperfurfaces which are {\it canonical}, i.e., such
that its canonical map is birational (see \ref{hypersurfaces})

\begin{theorem}
Let $f \colon X\longrightarrow B$ be a surjective flat morphism from a
$\Q$-factorial projective  $n$-fold $X$ to a smooth complete curve
$B$. Suppose that the fibration is relatively minimal and that the
general fibres $F$ are minimal canonical varieties (of dimension
$n-1$) with $p_g=n+1$, $K_F^{n-1}=n+2$, such that its canonical
image has at most log-terminal singularities. Then

$$
K_f^n\geq \frac{n(n+2)}{(n+1)}\deg f_*\omega_f.
$$

\end{theorem}
Examples of such fibres $F$ are smooth hypersurfaces of
$\mathbb{P}^n$ of degree $n+2$.

It is worth mentioning that this theorem is the first result proving
lower bounds for the slope of fibrations of dimension higher than $3$.

Eventually we give a new evidence of the necessity of the stability assumption
in the C-H theorem (Remark \ref{controhno}).

\bigskip

Chapter 3 is devoted to the study of a particular type of fibred
surfaces, the so called {\it trigonal} fibrations (i.e., when the
general fibre is a trigonal curve).
An intensively studied  problem in the last decades  is to find of
lower bounds for the slope of fibred surfaces.  In general, the so
called {\it slope inequality} holds (\cite{X} and  \cite{C-H},
\cite{LS})
$$s(f) \geq4-\frac{4}{g}.$$
It is sharp and equality is satisfied only for certain kind of
{\em hyperelliptic} fibrations \cite {A-K}, \cite{LS}.

There are several reasons to conjecture that the gonality of the
general fibre of $f$ has an (increasing) influence on the lower
bound of the slope (see \cite{konnocliff}, \cite{B-S},
\cite{S-F} and Remark \ref{genericity}).
So the  next  natural problem in this framework is the one of
studying trigonal fibrations.
The main known results are the following.

\medskip

\noindent {\it (Konno, \cite{konnotrig})
If $f\colon X\longrightarrow B$ is a trigonal fibration of genus $g\geq 6$, then
\begin{equation}\label{trigk}
s(f)
\geq \frac{14(g-1)}{3g+1}.
\end{equation}}

\noindent {\it (Stankova-Frenkel, \cite{S-F} prop.9.2 and prop.12.3)
If $f\colon X\longrightarrow B$ is a trigonal semistable fibration, then
\begin{equation}\label{trigf}
s(f)\geq \frac{24(g-1)}{5g+1}.
\end{equation}
\noindent This bound is sharp, and if equality holds the general fibres
have  Maroni invariant $\geq 2$.

\noindent Moreover, if $g$ is even and the following conditions hold:

$\bullet$ the general fibres have  Maroni invariant $0$;

$\bullet$  the singular fibres are irreducible and  have only certain kind of singularities;

\noindent then the slope satisfies the  bound
\begin{equation}\label{general}
s(f)\geq \frac{5g-6}{g}.
\end{equation}
}
Harris and Stankova-Frenkel conjecture (Conjecture 12.1 of \cite{S-F})  bound (\ref{general})
to hold without the extra condition on singular fibres.
It has to be remarked that the  bounds (\ref{trigf}) and (\ref{general}), although better than (\ref{trigk}),
hold only for semistable fibrations i.e. for fibred surfaces such that all the fibres are semistable curves
in the sense of Deligne-Mumford; this is, from the point of view of fibred surfaces, a strong restriction.
Indeed, starting from any fibred surface, one can construct a semistable one by the process
of semistable reduction, but the slope cannot be controlled through this process, as shown  in
\cite{T1}.

\medskip

The main result of Chapter 3 is the following (Theorem
\ref{slopetrigonal}):

\begin{theorem}\label{main}
Let $f\colon S \longrightarrow B$ be a relatively minimal fibred surface
such that the general fibre $C$ is a trigonal curve of even genus $g\geq 6$
and the general fibre has  Maroni invariant $0$.
Then the slope satisfies inequality (\ref{general}).
\end{theorem}

Observe that we are not assuming $f$ to be semistable. In
particular, we give a positive answer to the Harris-Stankova-Frenkel conjecture.
Moreover, in Theorem \ref{slopetrigonal}, we prove at the same time that   (\ref{general}) holds
for any  fibration of genus $6$  whose general fibres have a $g^2_5$, thus  improving the bound
proved by Konno in \cite{konnotrig}, which is $96/25$.

This result can be seen as a first step when searching for an
increasing dependence of the slope from the gonality of the
general fibres. The assumption on the Maroni invariant assures
that the fibres are general in the locus of trigonal curves,
consistently with the conjectures (see Remark \ref{genericity}).

\medskip

We prove this theorem applying the C-H method to a fibred $3$-fold naturally
associated to the fibred surface; indeed  the slope of $f$ is related to to the one of the
relative quadric-hull  $W\longrightarrow B$ of the
trigonal fibration $f\colon S\longrightarrow B$ (cf.
\cite{konnotrig}, \cite{Bcan}), for a suitable line bundle on it.
In the case of Maroni invariant $0$, the general fibre of the hull
is $\pr^1\times \pr^1$, embedded as a surface of minimal degree in
$\pr^{g-1}$, this embedding being GIT semistable by a result of
Kempf (\cite{K}).

\bigskip

\noindent{\bf Acknowledgements} We wish to thank Maurizio Cornalba, Andreas Leopold Knutsen and Andrea Bruno
for many helpful conversations on this topic.

\section{The Cornalba-Harris Method and the slope of fibrations}

We work over the complex field $\cp$.
Let $X$ be a variety (an integral separated scheme of finite type over $\cp$),
with a linear system $V\subseteq H^0(X,\cL)$, for some line bundle $\cL$ on $X$.
Fix $h\geq 1$ and call $G_h$ the image of the natural homomorphism
\begin{equation}\label{sim}
H^0(\mathbb P^{s}, \mathcal O_{\mathbb {P}^{s}}(h))=
\sym^hV\stackrel{\varphi_h}{-\!\!\!\longrightarrow}H^0(X, \cL^h).
\end{equation}
Set $N_h=\dim G_h$ and take exterior powers
$
\wedge^{N_h}\sym^hV\stackrel{\wedge ^{N_h}\varphi_h}{-\!\!\!\longrightarrow}
\wedge^{N_h}G_h=\det G_h.
$
We can see $\wedge ^{N_h}\varphi_h$ as a well-defined element of
$\pr (\wedge^{N_h}\sym^hV^{\vee})$.

With the above notations, we call $\wedge ^{N_h}\varphi_h\in
\pr(\wedge^{N_h} \mbox{Sym}^hV^{\vee})$, {\em the  generalized
$h$-th Hilbert point associated to the couple  $(X, V)$}.
If $V$ induces an embedding, then for $h\gg 0$ the homomorphism $\varphi_h$
is surjective  and
it is the classical $h$-th Hilbert point.

Consider the  standard representation  $SL(s+1,\cp)\rightarrow SL(V)$;
we get  an induced natural action of $SL(s+1,\cp)$ on  $\pr(\wedge^N \mbox{Sym}^hV^{\vee})$,
and we can introduce the associated notion of GIT (semi)stability:
we say that the $h$-th generalised Hilbert point of the couple  $(X, V)$
is  semistable (resp. stable) if it is GIT semistable (resp. stable) with
 respect to the natural $SL(s+1, \cp)$-action.

We say that $(X,V)$ is  {\em generalised Hilbert stable
(resp. semistable) if its generalised  $h$-th Hilbert point is stable (resp. semistable)
for infinitely many integers $h>0$.}

\smallskip

We  state a generalised version of the Cornalba-Harris theorem.

\begin{theorem}\label{slopecornalbaharris}[\cite{LS}, Theorem 1.5]
Let $X$ be a variety of dimension $n$ and
$\phi \colon X \longrightarrow B$ a  flat morphism over a  curve, and call $F$ a general fibre.
Let $\cL$ be a line bundle on $X$.
Let $h$ be a positive integer, and assume that
 $(F, \vert \cL_{\vert F} \vert)$ has semistable generalised $h$-th Hilbert point.
 Consider a vector subbundle $\cG_h$ of  $\phi_*\cL^h$ such that $\cG_h$ contains the image of the
morphism of sheaves
$$\sym^h\phi_*\cL\longrightarrow \phi_*\cL^h,$$
and coincides with it at general $t\in B$.

Then the line bundle
$$\cF_h:=\left(\det \phi_*\cL^{\phantom{1}} \right)^{-hN_h}\otimes \left(\det\cG_h\right)^r, $$
where $r:=h^0(F, \cL_{\vert F})$, and $N_h:=\rk \cG_h$,  is effective.

\end{theorem}

\begin{corollary}\label{CH2}
With the assumptions of Theorem \ref{slopecornalbaharris}, suppose moreover that
$X$ is pure dimensional, $\phi$ is proper, and that

(1) the linear system $\vert \cL_{\vert F} \vert$ induces an embedding of the general fibre $F$;

(2)  the sheaves $R^i\phi_*\cL^h$ vanish for $i>0$ for $h$ large enough\footnote{This happens for instance if $\cL$ is $\phi$-ample.}.

\noindent Then, the following inequality holds
\begin{equation}\label{equaz2}
L^n \geq n\frac{(L_{\vert F})^{n-1}}{h^0(F,\cL_{\vert F})}\deg \phi_*\cL.
\end{equation}
\end{corollary}
\begin{proof}
By the first assumption, we can apply Theorem \ref{slopecornalbaharris} with $\cG_h=\phi_*\cL^h$.
Hence,  for infinitely many $h>0$ the line bundle $\cF_h$ is effective.
Now, under our assumptions $\deg \cF_h$ is a degree $n$ polynomial in $h$ with
coefficients in the rational Chow ring $CH^1(X)_\Q$.
Its leading coefficient has to be pseudo-effective, hence to have non-negative degree.
The statement now follows from an intersection-theoretical computation.
Indeed,  the Riemann-Roch formula for singular varieties (cf. \cite{fulton}, Corollary 18.3.1)
implies the following expansions:
$$N_h= \frac{(L_{\vert F})^{n-1}}{(n-1)!}+ O(h^{n-2}),$$
and
$$\deg \phi_* \cL^h= \deg \phi_! \cL^h=\frac{h^n L^n}{n!}+ O(h^{n-1}),$$
because the higher direct images vanish by the second assumption.
\end{proof}

\begin{corollary}\label{CH3}
With the assumptions of Theorem \ref{slopecornalbaharris}, suppose moreover that
$X$ is pure dimensional, $\phi$ is proper,  and that for large enough  $h$

(1)  the linear system $\vert \cL_{\vert F} \vert$ induces a finite rational map on the image of $F$;

(2)  the vector bundle $\phi_*\cL^h$ is nef (i.e. every quotient has non-negative degree);

(3)  the sheaves $R^i\phi_*\cL^h$ vanish for $i>0$.

\noindent Then the following inequality holds
\begin{equation}\label{equaz3}
L^n \geq n\frac{d}{h^0(F,\cL_{\vert F})}\deg \phi_*\cL,
\end{equation}
where $d$ is the degree of the image  $\phi(F)\subseteq \pr^{r}$.
\end{corollary}
\begin{proof}
Let $h$ be large enough. By the nef-ness assumption on $\phi_*\cL^h$, the degree of
$\cF_h$ is smaller or equal to the degree of
$\left(\det \phi_*\cL^{\phantom{1}} \right)^{-hN_h}\otimes \left(\det\phi_*\cL^h\right)^r.$
Then the statement follows applying Riemann-Roch for singular varieties
as in the previous corollary,  observing that (by the first assumption)
$N_h=d h^{n-1}/(n-1)!+O(h^{n-2})$.
\end{proof}


In particular,  using  the relative canonical divisor,
we can obtain the following
result on the slopes of families of certain canonical varieties.

\begin{remark}\label{canonical}
Let $\phi \colon X \longrightarrow B$ be a fibration of a normal $\Q$-factorial variety
with at most canonical singularities over a curve.
Under these assumptions $K_X$ (and $K_\phi=K_X-\phi^* K_B$) is  a Weil, $\Q$-Cartier divisor.
We can consider its associated divisorial sheaves $\omega_X$ and $\omega_\phi$.
Suppose that the canonical sheaf  $ \omega_X $ is $\phi$-nef, and that
on a general fibre $F$ the canonical divisor   $\omega_F={\omega_\phi}_{\vert F}$ induces
a  Hilbert semistable map which is finite on the image of $F$.
Then the following inequality holds
$$
K_\phi^n\geq n\,\frac{d}{p_g(F)}\deg \phi_*\om_\phi,
$$
where $p_g(F)=h^0(F, K_F)$ and $d$ is the degree of the canonical image of the general fibre
$F$ in $\pr^{p_g(F)-1}$.
In particular, if $\omega_F$ induces a birational morphism, $d=K_F^{n-1}$.

Indeed, we can apply Corollary \ref{CH3} to the relative canonical sheaf:
$\cL=\om_\phi$.
The second assumption is satisfied by \cite{Viehweg}, while
the third one derives from the relative nefness of $\omega_X$, using the relative
version of Kawamata-Viehweg vanishing Theorem (see for instance \cite{KMM}, Theorem 1.2.3).
\end{remark}

\bigskip

Although it is difficult to check the stability assumption for
varieties of dimension bigger than 1, by a result of Tian we have

\medskip
\noindent {\it (Tian \cite{Tian}) Any normal hypersurface $F \subseteq
\mathbb{P}^N $ of degree $d \geq N+2 $ with only log-terminal
singularities is Hilbert stable.}

\medskip

We will say that a variety is {\it canonical} if its canonical map
is birational onto its image. Hence se can state

\begin{theorem}\label{hypersurfaces}
Let $\phi \colon X\longrightarrow B$ be a surjective flat morphism
from a $\Q$-factorial projective  $n$-fold $X$ to a smooth
complete curve $B$. Suppose that $K_\phi$ is $\phi$-nef and that
the general fibres $F$ are minimal canonical varieties (of
dimension $n-1$) with $p_g=n+1$,  $K_F^{n-1}=n+2$ whose
 canonical image has at most log-terminal singularities. Then

$$K_\phi^n\geq \frac{n(n+2)}{n+1}\deg \phi_*\omega_\phi.
$$

\end{theorem}
\begin{proof} It follows straightforward from the argument of Remark \ref{canonical} and Tian's
theorem.
\end{proof}
For instance a one-parameter family of surfaces with $p_g=4$ $q=0$ and $K^2=5$
such that the general fibre is of type $(I)$ in Horikawa classification (\cite{hor}, Theorem 1, sec.1)
satisfies the conditions of the above theorem; indeed these surfaces
have base-point-free birational canonical map, and their canonical image
is a $5$-ic surface in $\pr^3$ with at most rational double points.

\begin{remark}\label{controhno}
We can now give a new  example that show the fact that the stability
condition in the C-H method is necessary, in addition to the one
given by Morrison in section 3 of \cite{C-H}.

In \cite{ohno}, Example on page 664, a fibred $3$-fold
$\phi\colon T\longrightarrow B$ is
constructed fitting in the following diagram

$$
\xymatrix{
T  \ar[r]^\pi \ar[d]_\phi & W \ar[r]^\beta \ar[dl]^\alpha & {\mathbb P}_B(\phi_*\om_\phi)\ar[dll]\\
B& &\\}
$$

\noindent  such that

\begin{itemize}

\item the general fibre of $\phi$ is a surface of general type with
$p_g=4$, $q=0$ and $K^2_F=4$, and such that its canonical map is a
degree 2 base point free map on to a quadric cone in
$\mathbb{P}^3$.

\item the map $\pi$ is a smooth double cover of a
$\mathbb{P}^2$-bundle $W$ over $B$ such that $
 \phi_*\omega_\phi=
  \alpha_*(\omega_\alpha\otimes \cL)\oplus \alpha_*\omega_{\alpha}=
 \alpha_*(\omega_\alpha\otimes \cL)$ (because
$\alpha_*\omega_{\alpha}=0$, being the generic fibre of $\alpha$ a
rational surface)

\item the composition $\beta \circ \pi$ is the relative canonical
map of $\phi$

\item the slope of $\phi$ is $20/7$ (see \cite{ohno} page 665 with
e=2).

\end{itemize}

On the other hand, as $\omega_\alpha \otimes \cL$ induces on the
general fibres of $\alpha$ the natural map as a quadric cone in $
\pr^3$, by the same argument used by Konno in \cite{konnotrig},
Lemma 1.1, we can conclude that $R^i\alpha_*(\omega_\alpha\otimes
\cL)^h=0$ for $i, h>0$.

It is well known that the quadric cone is Hilbert unstable.
Assume nevertheless that we could apply Theorem \ref{slopecornalbaharris} to
$\omega_\phi$.
Consider the morphism of sheaves
$$
\sym^h \phi_*\omega_{ \phi}\longrightarrow \phi_*\omega_{ \phi}^h.
$$
We can choose as $\cG_h$ (in the notations of Theorem
\ref{slopecornalbaharris}) the sheaf
$\alpha_*(\omega_\alpha\otimes \cL)^h$. Computing now the degree
$3$ coefficient  of $\deg \cF_h$, we obtain
$$
\frac{K_{ \phi}^3}{2}=\pi^*(K_\alpha+L)^3\geq \frac{3}{2}\deg
{\phi}_*\omega_{\phi}.
$$

\noindent and hence the slope would be at least 3, a
contradiction.
\end{remark}


\section{The slope of trigonal fibrations}

Let $f:S \longrightarrow B$ be a relatively minimal fibred surface
such that the general fibre $C$ is a trigonal curve of genus $g$.

\begin{remark}\label{maroni}
If $C$ is a trigonal curve of genus $g$, it is a well known fact (see for instance \cite{Mar} and \cite{S-D})
that its canonical image lives on a Hirzebruch surface
$\F_c=\mathbb{P}(\cO_{\mathbb{P}^1}\oplus \cO_{\mathbb{P}^1}(c))$
embedded in $\mathbb{P}^{g-1}$ as a surface of minimal degree.
The surface $\F_c$  is the intersection of quadrics containing the canonical image of $C$ in $\pr^{g-1}$;
from a more geometric point of view, it is  the rational normal scroll
generated by the lines spanned by the divisors in the $g^1_3$ on the canonical image of $C$.
The number $c$ is called the {\it Maroni invariant of $C$};
it has the same parity of $g$, and satisfies the following inequalities:
$$
 \frac{g-2}{3}\leq c\leq \frac{2g-2}{3}.
$$

It has been shown in \cite{M-S} that  the Maroni invariant is an upper semicontinuos function on the trigonal locus $\mathcal D_3$,
and hence  a general genus $g$  trigonal curve  has Maroni invariant $0$ (resp.  $1$)
if  $g$ is even (resp. odd).
The locus of points in $\mathcal D_3$ corresponding to curves with Maroni invariant $>1$  has codimension $1$
the even case, while it has strictly bigger codimension in the odd one.
\end{remark}

We can extend the construction mentioned above on the fibres of $f$ to a relative
setting, using the so-called relative hyperquadric hull (see e.g. \cite{konnotrig} and \cite{Bcan}).
Consider the relative canonical image of $S$:

$$\xymatrix{
S \ar@{-->}[r]^-\psi \ar[d] &Y \subseteq
{\mathbb{P}}_B({f_*\om_f})={\mathbb{P}}
\ar[dl]^\varphi \\
B}$$


\noindent If $\cA \in  \pic B$ is ample enough it can be easily checked that
we have an epimorphism
$$\xymatrix{
H^0({\mathcal {J}}_{Y,{\mathbb{P}}}(2)\otimes
\varphi^{\ast} (\cA)) \ar@{->>}[r] & H^0({\mathcal {J}}_{F,
{\mathbb{P}}^{g-1}}(2)).}$$

Let $W_0$ be the horizontal irreducible component of the base locus
of the linear system  on $\mathbb{P}$ given by the sections of
 $H^0({\mathcal {J}}_{Y,{\mathbb{P}}}
(2 ) \otimes {\varphi}^{\ast}(\cA)))$. Since the
general fibre $C$ is trigonal,  $W_0$ is a threefold fibred over $B$
by rational surfaces of minimal degree.
Notice moreover that for $g\geq 5$ the singular locus of $W_0$ is contained in a finite
number of fibres.
Let $W$ be a desingularization of $W_0$ and let $L$ be the pull-back of the
tautological divisor of $\mathbb{P}$ to $ W$.
We will call ${W}$ the {\em relative quadric hull} associated to $f$
and denote by $\phi \colon { W} \longrightarrow B$ the
induced fibration.
The fibre of $\phi$  over general  $t\in B$ coincides with the one of $W_0$ (hence
it is the rational normal scroll
associated to the fibre of $S$ over $t$).

The main facts we will need about the divisor  $L$ have been proved by Konno
in \cite{konnotrig} (cf. Lemma 1.1 and Lemma 1.2).
\begin{proposition}[Konno]\label{Konno}

\quad

i) $\phi _* \cO_{W}(L)=f_*\of$;

ii) $R^p\phi_* \oo_W (hL)=0$ for $p,h>0$;

iii) $K^2_f \geq 2 \chi _f + L^3 $.
\end{proposition}
We can now state the main result of this chapter.

\begin{theorem}\label{slopetrigonal}
Let $f\colon S \longrightarrow B$ be a relatively minimal fibred surface
such that the general fibre $C$ is either:

$\bullet$  a trigonal curve of even genus $g\geq 6$
and zero Maroni invariant;

$\bullet$ a curve of genus $6$ with a $g^2_5$.

\noindent Then
$$s_f \geq \frac{5g-6}{g},$$
and this bound is sharp.
\end{theorem}

\begin{proof}
Using the relative quadric hull associated to $f$,
the general fibre $F$ of $\phi\colon {W} \longrightarrow B$ is just
$\pr^1\times\pr^1$ in the trigonal case and $\pr^2$ in the case of a plane quintic.
The restriction of $L$ to $F$ induces a complete embedding of $F$ in
$\pr^{g-1}$ as a surface of minimal degree.
This is Hilbert semistable according, for instance,  to a result of Kempf (cf. \cite{K} cor. 5.3);
Moreover, $\oo_W(L)$ has no higher relative cohomology by  Proposition \ref{Konno}, $(ii)$.
We can therefore apply Corollary \ref{CH2} and conclude that
$$L^3 \geq 3\frac{g-2}{g} \, {\rm deg} \phi_*\cO_{W}(L).$$
The statement now  follows using inequality $(i)$ and $(iii)$ of Proposition \ref{Konno}.

As for the sharpness of this bound, we refer to Example \ref{sharp} below.
\end{proof}

\begin{example}\label{sharp}
Using a construction of Tan \cite{Tanslope}, we can prove that any trigonal curve with Maroni invariant
{\em strictly} smaller than $(g+2)/9$ can be realized as the fibre of a semistable fibration over $\pr^1$ with
slope $(5g-6)/g$.
Let $C$ be a trigonal curve with Maroni invariant $c$.
Recall that $\mbox{Pic} \F_c=\Z[\ell]\oplus \Z[f]$, where $\ell$ is the negative section ($\ell^2=-c$) and
$f$ is a fibre of the ruling $\F_c\rightarrow \pr^1$.
The class of $C\subset \F_c$ is $3\ell+kf$, where
$k=(g+2+3c)/2$.

As proved for instance in \cite{Har}V.Cor 2.18, the linear system $\mid 3\ell + kf\mid$ is very ample if and only if
$k>6c$, that is $9c<g+2$.
In this case, we can choose a general pencil in $\mid C\mid$.
It has $C^2=( 3\ell + kf)^2=6k-9c=3g+6$ base points.
Let $S$ be the blow up of $\F_c$ in these base points, and $f\colon  S\longrightarrow \pr^1$ the induced fibration.
Computing the relative invariants, we obtain $K_f^2=5g-6$ and $\chi_f=g$.

In \cite{konnocliff}, Example 4.6., other examples reaching the bound are provided,
satisfying the condition that  the bundle $f_*\omega_f$ is semistable.
\end{example}

\begin{remark}
The higher Maroni invariant cases cannot be treated with the C-H method.
Recall that the general fibre of $W$ is an Hirzebruch surface
$\F_c$ embedded in $\pr^{g-1}$ by the divisor $D=\ell+\frac{g+c-2}{2}f$;
$D$ is a ``good'' divisor in the notation of  \cite{Mor}, and by Theorem 6.5 of the same paper,
the associated embedding is Chow unstable (hence Hilbert unstable) if and only if $c>0$.

On the other hand, Xiao's method has been applied to this setting (regardless to the Maroni invariant) by Konno
in \cite{konnotrig},  and leads to the bound (\ref{trigk}).

This seems to suggest that the two methods of Cornalba-Harris and of Xiao,
while being surprisingly similar in the case of fibred surfaces (cf \cite{B-S}),
become substantially different when applied to fibrations whose total space has
dimension $\geq 3$.
\end{remark}

\begin{remark}\label{genericity}
As it is well known, gonality provides a stratification of the moduli space of smooth curves ${\mathcal M}_g$.
Indeed,  the loci
$$\mathcal D_k:=\overline{\left\{[C]\in {\mathcal M}_g \mbox{ such that } C \mbox{ has a } g^1_k\right\}}\subseteq {\mathcal M}_g$$
are closed subsets of ${\mathcal M}_g$ of decreasing codimension as $k$ goes from $2$ to $[(g+3)/2]$.
The curves with maximal gonality $[(g+3)/2]$ form an open set.
It has been proved (\cite{A-K} and \cite{Ha-Mu} in the semistable case) that
if $f\colon S\longrightarrow B$ is a fibred surface of  odd genus and such that the general fibres have maximal gonality,
then
$s(f)\geq 6(g-1)/(g+1)$.
Moreover, the slope inequality $s(f)\geq 4(g-1)/g$, that holds for any fibres surface, is an equality only for some hyperelliptic
fibrations.
It seems therefore natural to conjecture an increasing bound for the slope of fibred surfaces
depending on the gonality of the general fibres.
A natural guess  would be that the slope of non-hyperelliptic fibred surfaces should satisfy at least
the bounds for trigonal fibrations (\ref{trigk}).
This is however false, as observed for instance in \cite {BPhD} and \cite{S-F}:
the easiest counterexamples are provided by
bielliptic surfaces of arbitrarily large genus,  with slope $4$.

The right question to ask when looking for a bound increasing with gonality
seems to be  that the fibre are ``general'' in the $k$-gonal locus:
 (see in particular Conjecture 13.3 of \cite{S-F}).
From this point of view, the bound (\ref{general}) could be the first step of the desired sequence.

It  is worth mentioning that Konno proves the same bound (\ref{general}) in \cite{konnocliff}
(corollary 4.4) under the assumptions that the fibration is non-hyperelliptic,
and  that  {\it $f_*\of $ is a semistable vector bundle.}
The assumption of semistability for $f_*\of $ is difficult to interpret; it would be very interesting
to understand whether  it is connected with some kind of  ``genericity'' of
the general fibre.
\end{remark}

\addcontentsline{toc}{section}{References}

\bigskip
\noindent Miguel \'Angel Barja,  Departament de Matem\`atica  Aplicada I, Universitat Polit\`ecnica de Catalunya,
 ETSEIB Avda. Diagonal, 08028 Barcelona (Spain).\\
E-mail: \textsl{Miguel.Angel.Barja@upc.edu}

\bigskip
\noindent Lidia Stoppino, Dipartimento di Matematica, Universit\`a di Pavia, Via Ferrata 1  27100, Pavia  (Italy).
E-mail: \textsl {lidia.stoppino@unipv.it}.


\begin{thebibliography}{99}


\bibitem{A-K} T. Ashikaga and K. Konno, \emph{Global and local properties of pencils of algebraic curves}, Algebraic geometry 2000, Azumino (Hotaka), Adv. Stud. in Pure Math. {\bf 36}  Math. Soc. Japan, Tokyo, 2002, 1-49.

\vspace{-0.1cm}


\bibitem{BPhD} M. A. Barja, \emph{On the slope and geography of fibred surfaces and threefolds}, Ph. D. Thesis, Univesity of Barcelona, 1998.

\vspace{-0.1cm}

\bibitem{Barja3folds} M. A. Barja, \emph{On the slope of fibred threefolds}, Internat. J. Math. {\bf 11} n.4 (2000), 461-491.


\vspace{-0.1cm}

\bibitem{B-S} M. A. Barja and L. Stoppino, \emph{Linear stability of projected canonical curves with applications to the slope of fibred surfaces}, preprint arXiv:math/0612030, to appear in  J. Math. Soc. Japan.

\bibitem{B-Z} M. A. Barja and F. Zucconi, \emph{On the slope of fibred surfaces}, Nagoya Math. J. {\bf 164} (2001), 103-131.

\vspace{-0.1cm}

\bibitem{Bcan} M. A. Barja {\em Numerical bounds of canonical varieties}, Osaka J. Math. {\bf 37 (3)} (2000), 701-718.



\vspace{-0.1cm}

\bibitem{C-H} M. Cornalba, J. Harris, \emph{Divisor classes associated to families of stable varieties,
with applications to the moduli space of curves.} Ann. Sc. Ec. Norm. Sup. {\bf 21 (4)} (1988), 455-475.


\vspace{-0.1cm}

\bibitem{fulton} W. Fulton, \emph{Intersection Theory, second edition}, Springer-Verlag, 1998.


\vspace{-0.1cm}


\bibitem{Ha-Mu} J. Harris, D. Mumford, {\em On the {K}odaira dimension of the moduli space of curves. With an appendix by William Fulton},
Invent. Math., {\bf 67(1)}, (1982), 23-88.

\vspace{-0.1cm}

\bibitem{Har} R. Hartshorne Algebraic geometry. GTM, No. 52. Springer-Verlag, New York-Heidelberg, 1977.

\vspace{-0.1cm}

\bibitem{hor} E. Horikawa, {\em On deformations of quintic surfaces},  Invent. Math. {\bf 31} (1975), 43-85.

\vspace{-0.1cm}




\bibitem{KMM} Y. Kawamata and K. Matsuda and K. Matsuki, {\em Introduction to the {M}inimal {M}odel {P}roblem}, Adv. Stud. in Pure Math., {\bf 10}, (1987), 283-360.

\vspace{-0.1cm}

\bibitem{K} G. R. Kempf, {\em Instability in invariant theory}, Ann. of Math. {\bf 108}, (1978), 299-316.

\vspace{-0.1cm}

\bibitem{konnocliff} K. Konno, \emph{Clifford index and the slope of fibered surfaces}, J. Algebraic Geom. {\bf 8 (2)} (1999), 207-220.

\vspace{-0.1cm}

\bibitem{konnotrig} K. Konno, \emph{A lower bound of the slope of trigonal fibrations}, Internat. J. Math. {\bf 7 (1)} (1996), 19-27.

\vspace{-0.1cm}



\bibitem{Mar} A. Maroni, {\em Le serie lineari speciali sulle curve trigonali}, Annali di  Matematica,  serie IV, {\bf 25} (1946) 341-354.

\vspace{-0.1cm}

\bibitem{M-S} G. Martens, F. O. Schreyer, {\em Line bundles and syzygies of trigonal curves}, Abh. Math. Sem. Univ. Hamburg, {\bf 56} (1986), 169-189.

\vspace{-0.1cm}

\bibitem{Mor} I. Morrison, {\em Projective stability of ruled surfaces}, Invent. Math. {\bf 56 (3)} (1980),  269-304.

\vspace{-0.1cm}

\bibitem{Mum} D. Mumford, \emph{Stability of projective varieties}, L'Ens. Math. {\bf 23} (1977), 39-110.

\vspace{-0.1cm}

\bibitem{ohno} K. Ohno, \emph{Some inequalities for minimal fibrations of surfaces of general type over curves}, J. Math. Soc. Japan {\bf 44 (4)} (1992), 643-666.




\vspace{-0.1cm}



\bibitem{S-D} B. Saint-Donat, {\em On Petri's analysis of the linear system of quadrics through a canonical curve} Math. Ann. {\bf 206} (1973), 157-175.

\vspace{-0.1cm}

\bibitem{S-F} Z. E. Stankova-Frenkel, \emph{Moduli of trigonal curves}, J. Algebraic Geom. {\bf 9 (4)} (2000), 607-662.


\vspace{-0.1cm}

\bibitem{LS} L. Stoppino, \emph{Inequalities for fibred surfaces via GIT}, to appear in Osaka Math. J., Vol. 45, No. 4 (2008), preprint math.AG/0411639.

\vspace{-0.1cm}

\bibitem{T1} S. Tan, \emph{On the invariants of base changes of pencils of curves I}, Manuscripta Math. {\bf 84} (1994), 225-244. \emph{On the invariants of base changes of pencils of curves II}, Math. Z. {\bf 222} (1996), 655-676.

\vspace{-0.1cm}

\bibitem{Tanslope} S. Tan, \emph{On the slopes of the moduli spaces of curves}, International J. of Math. {\bf 9} (1998), 119-127.

\vspace{-0.1cm}

\bibitem{Tian} G. Tian, {\em The $k$-energy on hypersurfaces and stability}, Communications in  Analysis and Geometry, {\bf 2 (2)}  (1994), 239-265.

\vspace{-0.1cm}

\bibitem{Viehweg} E. Viehweg, \emph{Quasi-projective moduli for polarized manifolds}, Springer-Verlag, Ergebnisse der Mathematik und ihrer Grenzgebiete vol. 30, 1995.

\vspace{-0.1cm}

\bibitem{X} G. Xiao, \emph{Fibred algebraic surfaces with low slope.} Math. Ann. {\bf 276} (1987), 449-466.

\end{thebibliography}
\end{document}